\documentclass{article}

\usepackage{amsmath}
\usepackage{amsfonts}
\usepackage{amsthm}

\title{Gr\"{o}bner basis and singular locus of Lauricella's hypergeometric differential equations}

\author{Hiromasa Nakayama (Kobe University/JST CREST) 
\thanks{nakayama@math.kobe-u.ac.jp}
}

\newcommand{\C}{\mathbb{C}}
\newcommand{\R}{\mathbb{R}}
\newcommand{\inn}{{\rm in}}
\newcommand{\p}{\partial}

\newcommand{\Dan}{\mathcal{D}}
\newcommand{\Zzero}{\mathbb{Z}_{\geq 0}}
\newcommand{\OrdZeroOne}{{<_{({\bf 0}, {\bf 1})}}}
\newcommand{\rest}{{\rm rest}}
\newcommand{\WZeroOne}{({\bf 0}, {\bf 1})}
\newcommand{\innZeroOne}{\inn_{\WZeroOne}}

\newtheorem{proposition}{Proposition}
\newtheorem{theorem}[proposition]{Theorem}
\newtheorem{lemma}[proposition]{Lemma}

\newtheorem{remark}[proposition]{Remark}

\begin{document}
  \maketitle

\section{Introduction}

Lauricella's hypergeometric series $F_A, F_B, F_C$ are defined by 
\begin{align*}
& F_A(a, b_1, \ldots, b_m, c_1, \ldots, c_m ; x_1, \ldots, x_m) = \\ 
& \quad \quad \sum_{n_1, \ldots, n_m \in \Zzero} 
\frac{(a)_{n_1 + \cdots + n_m} (b_1)_{n_1} \cdots (b_m)_{n_m}}
{(c_1)_{n_1} \cdots (c_m)_{n_m} (1)_{n_1} \cdots (1)_{n_m}} x_1^{n_1} \cdots x_m^{n_m}, 
\end{align*}
\begin{align*}
& F_B(a_1, \ldots, a_m, b_1, \ldots, b_m, c ; x_1, \ldots, x_m) = \\
& \quad \quad \sum_{n_1, \ldots, n_m \in \Zzero} 
\frac{(a_1)_{n_1} \cdots (a_m)_{n_m} (b_1)_{n_1} \cdots (b_m)_{n_m}}
{(c)_{n_1 + \cdots + n_m} (1)_{n_1} \cdots (1)_{n_m}} x_1^{n_1} \cdots x_m^{n_m},
\end{align*}
\begin{align*}
& F_C(a, b, c_1, \ldots, c_m ; x_1, \ldots, x_m) = \\
& \quad \quad \sum_{n_1, \ldots, n_m \in \Zzero} 
\frac{(a)_{n_1+ \cdots +n_m} (b)_{n_1 + \cdots + n_m}}
{(c_1)_{n_1} \cdots (c_m)_{n_m} (1)_{n_1} \cdots (1)_{n_m}} x_1^{n_1} \cdots x_m^{n_m}. 
\end{align*}
Here, 
$a, b, c, a_i, b_i, c_i (i = 1, \ldots, m)$ are parameters and   
$c, c_i \notin \mathbb{Z}_{\leq 0}$.
Lauricella's series $F_A, F_B, F_C$ satisfy the following systems of 
differential equations respectively
$$
\ell_i^A \cdot F_A = 0,  
\ell_i^A = \theta_i (\theta_i + c_i - 1) - x_i(\theta_1 + \cdots +
 \theta_m + a)(\theta_i + b_i) ~ 
(i = 1, \ldots, m), 
$$
$$ 
\ell_i^B \cdot F_B = 0,  
\ell_i^B = \theta_i (\theta_1 + \cdots + \theta_m + c - 1) - 
x_i (\theta_i + a_i)(\theta_i + b_i) ~
(i = 1, \ldots, m), 
$$ 
$$
\ell_i^C \cdot F_C = 0, 
\ell_i^C = \theta_i(\theta_i + c_i - 1) - x_i (\theta_1 + \cdots +
\theta_m + a)(\theta_1 + \cdots + \theta_m + b) 
(i = 1, \ldots, m).
$$
Here, 
$\p_i = \frac{\p}{\p x_i}$ is the differential operator for $x_i$ and  
$\theta_i = x_i \p_i$ is Euler operator for $x_i$.

Let 
$D = \C[x_1, \ldots, x_m] \langle \p_1, \ldots, \p_m \rangle$
be the ring of differential operators with polynomial coefficients and 
$\Dan = \C\{x_1, \ldots, x_m\} \langle \p_1, \ldots, \p_m \rangle$ 
be the ring of differential operators with convergent power series coefficients.
We define $D$ ideals 
\begin{align*}
& I_A(m) = D \cdot \{\ell_i^A ~|~ i = 1, \ldots, m \}, \\
& I_B(m) = D \cdot \{\ell_i^B ~|~ i = 1, \ldots, m \}, \\ 
& I_C(m) = D \cdot \{\ell_i^C ~|~ i = 1, \ldots, m \}, 
\end{align*}
and $\Dan$ ideals
$$
\mathcal{I}_A(m) = \Dan \cdot \{\ell_i^A ~|~ i = 1, \ldots, m \}, \quad
\mathcal{I}_C(m) = \Dan \cdot \{\ell_i^C ~|~ i = 1, \ldots, m \}.
$$
We note that these ideals can be defined without any condition on parameters.
We will call these ideals Lauricella's systems of differential equations.

In this paper, 
we obtain Gr\"{o}bner bases for these ideals $I_B(m), \mathcal{I}_A(m), \mathcal{I}_C(m)$
with respect to some monomial orders without any condition on parameters.
By utilizing these Gr\"{o}bner bases and by Oaku's celebrating result on Gr\"{o}bner basis, 
characteristic varieties and singular locus \cite{OSSingLocus}, \cite{OakuSingLocus}, 
we will determine the singular locus of Lauricella's system of differential equations 
$I_t(m)$ where $t$ is $A$ or $B$ or $C$.

The singular locus under some conditions on parameters is known for these equations
(see the survey by Matsumoto \cite{M2}).
However, the singular locus without any condition on parameters has not been known.
We note that Hattori and Takayama \cite{HT} determined the singular locus of 
the system $I_C(m)$ recently without any assumption on parameters by utilizing
Gr\"{o}bner basis, syzygies and cohomological solutions.
Our method also utilizes Gr\"{o}bner basis, but it is simpler by considering Gr\"{o}bner 
bases in the ring $\Dan$ and can be applied to other Lauricella's systems.
We will use notations of \cite{SST} throughout this paper.

\section{Gr\"{o}bner basis for Lauricella's hypergeometric differential equations}
We derive a Gr\"{o}bner basis for the $D$ ideal $I_B(m)$.
\begin{theorem} \label{thm-gr-b}
Let $\xi_i$ be a commutative variable corresponding to $\p_i$.
We define a term order $\OrdZeroOne$ as follows.
The relation
$$ x_1^{\alpha_1} \cdots x_m^{\alpha_m} \xi_1^{\beta_1} \cdots \xi_m^{\beta_m} 
   \OrdZeroOne x_1^{\alpha_1'} \cdots x_m^{\alpha_m'} \xi_1^{\beta_1'} \cdots \xi_m^{\beta_m'} $$
holds if and only if one of the following cases holds:
\begin{enumerate}
\item $\beta_1 + \cdots + \beta_m < \beta_1' + \cdots + \beta_m'$ 
\item $\beta_1 + \cdots + \beta_m = \beta_1' + \cdots + \beta_m'$ and  
      $\alpha_1 + \cdots + \alpha_m < \alpha_1' + \cdots + \alpha_m'$
\item $\beta_1 + \cdots + \beta_m = \beta_1' + \cdots + \beta_m'$ and  
      $\alpha_1 + \cdots + \alpha_m = \alpha_1' + \cdots + \alpha_m'$ and \\
      $x_1^{\alpha_1} \cdots x_m^{\alpha_m} \xi_1^{\beta_1} \cdots \xi_m^{\beta_m} 
      <'x_1^{\alpha_1'} \cdots x_m^{\alpha_m'} \xi_1^{\beta_1'} \cdots \xi_m^{\beta_m'}$. 
      Here, $<'$ is a term order such as the lexicographic order.
\end{enumerate} 
Then, the set $\{\ell_1^B, \ldots, \ell_m^B \}$ is a Gr\"{o}bner basis 
for $I_B(m)$ with respect to $\OrdZeroOne$.
\end{theorem}
In order to prove the theorem, we use the Buchberger's criterion.
In other words, we prove that any S-pair is reduced to 0.
We will use the following lemma to simplify S-pairs.
\begin{lemma} \label{lem-cri-d}
Let $P, Q \in D$ and $<$ be a term order on $D$.
If the initial terms $\inn_<(P)$ and $\inn_<(Q)$ are relatively prime, 
the $S$-pair of $P$ and $Q$ $S_<(P, Q)$ is reduced to the commutator $-[P,Q]$.
\end{lemma}
\begin{proof}
We may assume that the coefficients of the initial terms of $P$ and $Q$ are 1 
without loss of generality.
Since the initial terms $\inn_<(P)$ and $\inn_<(Q)$ are relatively prime,
we have  
\begin{align*}
S_<(P, Q) &= (\inn_<(Q)(x,\p)) P - (\inn_<(P)(x,\p)) Q \\
          &= (Q - \rest_<(Q)) P - (P - \rest_<(P)) Q \\
          &= -\rest_<(Q) P + \rest_<(P) Q + QP - PQ \\ 
          &= -\rest_<(Q) P + \rest_<(P) Q - [P, Q]. 
\end{align*} 
When  
$\inn_<(P) = x_1^{\alpha_1} \cdots x_m^{\alpha_m} \xi_1^{\beta_1} \cdots \xi_m^{\beta_m}$, 
we define 
$$\inn_<(P)(x,\p) = x_1^{\alpha_1} \cdots x_m^{\alpha_m} \p_1^{\beta_1} \cdots \p_m^{\beta_m}$$
and   
$\rest_<(P) = P - \inn_<(P)(x,\p)$.
The $S$-pair $S_<(P, Q)$ is reduced to the commutator $-[P,Q]$ by $P$ and $Q$. 
\end{proof}
\begin{proof} {\it (of Theorem \ref{thm-gr-b})}
We need to show that the $S$-pair of $\ell_i^B, \ell_j^B ~ (1 \leq i < j \leq m)$ is reduced to $0$.
Since the initial terms $\inn_\OrdZeroOne(\ell_i^B) = x_i^3 \xi_i^2$ and 
$\inn_\OrdZeroOne(\ell_j^B) = x_j^3 \xi_j^2$ are relatively prime, 
we can use Lemma \ref{lem-cri-d}.
The commutator of $\ell_i^B$ and $\ell_j^B$ is 
\begin{align*}
[\ell_i^B, \ell_j^B] &= \ell_i^B \ell_j^B - \ell_j^B \ell_i^B  \\
 &= x_i (\theta_i + a_i) (\theta_i + b_i) \theta_j - x_j (\theta_j + a_j)(\theta_j + b_j) \theta_i \\
 &\xrightarrow[\ell_i^B, \ell_j^B]{*} \theta_i(\theta_1 + \cdots + \theta_m + c - 1) \theta_j 
   - \theta_j (\theta_1 + \cdots + \theta_m + c - 1) \theta_i = 0, 
\end{align*}
where
$\xrightarrow[\ell_i^B, \ell_j^B]{*}$ means the reduction by $\ell_i^B$ and $\ell_j^B$.
Since the commutator is reduced to $0$,
the $S$-pair $S_\OrdZeroOne(\ell_i^B, \ell_j^B)$ is reduced to $0$
by Lemma \ref{lem-cri-d}. 
By the Buchberger's criterion, 
the set $\{\ell_1^B, \ldots, \ell_m^B\}$ is a Gr\"{o}bner basis with respect to $\OrdZeroOne$.
\end{proof}
\begin{remark} \label{rem-gr-b}
In Theorem \ref{thm-gr-b}, 
we obtain a Gr\"{o}bner basis of $I_B(m)$ with respect to the term order $\OrdZeroOne$.
We are interested in the set of term orders for which the set of generators
$\{\ell_1^B, \ldots, \ell_m^B\}$ is a Gr\"{o}bner basis.
Let us determine the weight vector $w$ and the tie-breaker order such that 
the term order $<_w$ satisfies $\inn_{<_w}(\ell_i^B) = x_i^3 \xi_i^2$.
This condition yields the condition that the weight vector 
$w = (w_1, \ldots, w_m, w_{m+1}, \ldots, w_{2m}) \in (\R_{\geq 0})^{2m}$ 
satisfies
\begin{equation} \label{w-cone}
w_i > 0,~ w_{m+i} \geq 0,~ 2 w_i - w_k + w_{m+i} - w_{m+k} > 0 ~ 
(1 \leq k \leq m \text{ and } k \neq i)
\end{equation}
for $i = 1, \ldots, m$. 
We define 
$$ x_1^{\alpha_1} \cdots x_m^{\alpha_m} \xi_1^{\beta_1} \cdots \xi_m^{\beta_m} 
   <_w x_1^{\alpha_1'} \cdots x_m^{\alpha_m'} \xi_1^{\beta_1'} \cdots \xi_m^{\beta_m'} $$
if and only if one of the following cases holds:
\begin{enumerate}
\item $w_ 1 \alpha_1 + \cdots + w_m \alpha_m + w_{m+1} \beta_1 + \cdots + w_{2m} \beta_m 
       < w_ 1 \alpha_1' + \cdots + w_m \alpha_m' + w_{m+1} \beta_1' + \cdots + w_{2m} \beta_m'$ 
\item $w_ 1 \alpha_1 + \cdots + w_m \alpha_m + w_{m+1} \beta_1 + \cdots + w_{2m} \beta_m 
       = w_ 1 \alpha_1' + \cdots + w_m \alpha_m' + w_{m+1} \beta_1' + \cdots + w_{2m} \beta_m'$ 
      and 
      $x_1^{\alpha_1} \cdots x_m^{\alpha_m} \xi_1^{\beta_1} \cdots \xi_m^{\beta_m} 
      < x_1^{\alpha_1'} \cdots x_m^{\alpha_m'} \xi_1^{\beta_1'} \cdots \xi_m^{\beta_m'}$. 
      Here, $<$ is a term order such as the lexicographic order.
\end{enumerate} 
We can prove that the set of generators $\{\ell_1^B, \ldots, \ell_m^B\}$ 
is a Gr\"{o}bner basis with respect to the term order $<_w$. 
We note that the weight vector $(0, \ldots, 0, 1, \ldots, 1)$ lies in 
the closure of the cone (\ref{w-cone}) in the weight space.
The set of generators $\{\ell_1^B, \ldots, \ell_m^B\}$ is 
a Gr\"{o}bner basis with respect to the term order defined by 
the weight vector $(0, \ldots, 0, 1, \ldots, 1)$ and the tie-breaker $<_w$.
\end{remark}

Next, 
We derive a Gr\"{o}bner basis for $\Dan$ ideal $\mathcal{I}_A(m)$.  
\begin{theorem} \label{thm-gr-a}
We define a monomial order on $\OrdZeroOne'$ on $\Dan$ as follows.
The relation 
$$ x_1^{\alpha_1} \cdots x_m^{\alpha_m} \xi_1^{\beta_1} \cdots \xi_m^{\beta_m} 
   \OrdZeroOne' x_1^{\alpha_1'} \cdots x_m^{\alpha_m'} \xi_1^{\beta_1'} \cdots \xi_m^{\beta_m'}$$
holds if and only if one of the following case holds:
\begin{enumerate}
\item $\beta_1 + \cdots + \beta_m < \beta_1' + \cdots + \beta_m'$ 
\item $\beta_1 + \cdots + \beta_m = \beta_1' + \cdots + \beta_m'$ and 
      $\alpha_1 + \cdots + \alpha_m > \alpha_1' + \cdots + \alpha_m'$
\item $\beta_1 + \cdots + \beta_m = \beta_1' + \cdots + \beta_m'$ and
      $\alpha_1 + \cdots + \alpha_m = \alpha_1' + \cdots + \alpha_m'$ and \\
      $x_1^{\alpha_1} \cdots x_m^{\alpha_m} \xi_1^{\beta_1} \cdots \xi_m^{\beta_m} 
      <'x_1^{\alpha_1'} \cdots x_m^{\alpha_m'} \xi_1^{\beta_1'} \cdots \xi_m^{\beta_m'}$. 
      Here, $<'$ is a term order such as the lexicographic order.
\end{enumerate} 
Then, 
the set $\{\ell_1^A, \ldots, \ell_m^A\}$ is 
a Gr\"{o}bner basis for $\Dan$ ideal $\mathcal{I}_A(m)$ with respect to $\OrdZeroOne'$.
\end{theorem}
Before giving a proof, we note that
for the monomial order $\OrdZeroOne'$ in the ring $\mathcal{D}$, 
an analogous lemma as Lemma \ref{lem-cri-d} also holds.
\begin{proof}
We need to prove 
the $S$-pair of $\ell_i^A$ and $\ell_j^A ~ (1 \leq i < j \leq m)$ is reduced to $0$.
Since the initial terms $\inn_{\OrdZeroOne'}(\ell_i^A) = x_i^2 \xi_i^2$ and 
$\inn_{\OrdZeroOne'}(\ell_j^A) = x_j^2 \xi_j^2$ are relatively prime, 
we can use the analogous lemma as Lemma \ref{lem-cri-d}.
The commutator $[\ell_i^A, \ell_j^A] = 0$.
The $S$-pair $S_\OrdZeroOne'(\ell_i^A, \ell_j^A)$ is reduced to $0$.
By the Buchberger's criterion with respect to a monomial order $\OrdZeroOne'$ in $\Dan$
\cite{Castro0}, \cite[Th 1.4.]{OSSingLocus},  
the set $\{\ell_1^A, \ldots, \ell_m^A\}$ is a Gr\"{o}bner basis.
\end{proof}
For $\Dan$ ideal $\mathcal{I}_C(m)$, 
we can also derive a Gr\"{o}bner basis analogously. 
\begin{theorem}
The set $\{\ell_1^C, \ldots, \ell_m^C\}$
is a Gr\"{o}bner basis for $\Dan$ ideal $\mathcal{I}_C(m)$ with respect to $\OrdZeroOne'$.
\end{theorem}

\begin{remark}
We could not derive 
Gr\"{o}bner bases for $D$ ideals $I_A(m), I_C(m)$ and the $D$ ideal for $F_D$ 
with respect to the term order $\OrdZeroOne$.
Gr\"{o}bner bases for these ideals seem to be more complicated from computer experiments.
This is why we discuss on a Gr\"{o}bner basis in the ring $\mathcal{D}$ to study the singular locus. 
\end{remark}

\section{Singular locus of Lauricella's system $I_A(m), I_B(m)$}
Hattori and Takayama \cite{HT} determined the singular locus of Lauricella's system $I_C(m)$ 
by using Gr\"{o}bner basis, syzygies and cohomological solutions.
We determine the singular locus of Lauricella's system  $I_A(m), I_B(m)$ by using 
the obtained Gr\"{o}bner bases.

\subsection{Singular locus of Lauricella's system $I_B(m)$}
We compute singular locus of Lauricella's system $I_B(m)$.
By Theorem \ref{thm-gr-b}, 
a Gr\"{o}bner basis of $I_B(m)$ with respect to $\OrdZeroOne$ is $\{\ell_1^B, \ldots, \ell_m^B\}$.
We define a weight vector $\WZeroOne = (0, \ldots, 0, 1, \ldots, 1) \in {\mathbb{Z}}^{2m}$.
In other words, we set that the weight of $x_i$ is 0 and that of $\xi_i$ is $1$.
We define the $\WZeroOne$ initial form $\innZeroOne(P)$ of the differential operator $P$ by 
the sum of the terms in $P$ which has the highest $\WZeroOne$ weight.
In other words, when
$P = \sum_{\alpha, \beta \in (\Zzero)^m} c_{\alpha,\beta} x^\alpha \p^\beta \in D$, 
the $\WZeroOne$ initial form of $P$ is 
$$ \innZeroOne(P) = \sum_{\WZeroOne \cdot (\alpha, \beta) \text{ is maximum in } P} c_{\alpha, \beta} x^\alpha \xi^\beta.$$
For a $D$ ideal $I$, 
the $\WZeroOne$ initial form ideal is defined by the $\C[x,\xi]$ ideal   
$$ \innZeroOne(I) = \langle \innZeroOne(P) ~|~ P \in I \rangle.$$
By the property of Gr\"{o}bner basis with respect to $\OrdZeroOne$, 
the $\WZeroOne$ initial form ideal $\innZeroOne(I_B(m))$ are generated by 
$\innZeroOne(\ell_1^B),  \ldots, \innZeroOne(\ell_m^B)$.
The $\WZeroOne$ initial form of $\ell_i^B$ is 
$$ \innZeroOne(\ell_i^B) = x_i \xi_i \left(x_i(1-x_i)\xi_i + \sum_{1 \leq j \leq m, j \neq i} x_j \xi_j \right).$$
We denote the initial form by $L_i^B$.
By the Oaku's result \cite[Proposition 1 in Section 2]{OakuSingLocus} and Theorem \ref{thm-gr-b}, 
we have the following proposition.
\begin{proposition} \label{prop-char-b}
The characteristic variety for $D$ ideal $I_B(m)$ is 
${\rm Ch}(I_B(m)) = {\bf V}(L_1^B, \ldots, L_m^B).$
\end{proposition}
The singular locus of $I_B(m)$ is defined by 
$$ {\rm Sing}(I_B(m)) = \pi({\rm Ch}(I_B(m)) \setminus \{\xi_1 = \cdots = \xi_m = 0\}).$$
Here, $\pi$ is the projection 
$\C^{2m} \ni (x_1, \ldots, x_m, \xi_1, \ldots, \xi_m) \mapsto (x_1, \ldots, x_m) \in \C^m$.
In order to compute the singular locus, 
we need to compute the solution $(x_1, \ldots, x_m, \xi_1, \ldots, \xi_m)$ with 
$(\xi_1, \ldots, \xi_m) \neq (0, \ldots, 0)$ for 
$$ L_1^B = 0, \ldots, L_m^B = 0 $$ 
i.e., 
\begin{equation} \label{sing-eq-b}
x_i \xi_i = 0 \text{~ or ~} x_i(1-x_i) \xi_i + \sum_{1\leq k \leq m, k \neq i} x_k \xi_k = 0 \quad (i = 1, \ldots, m). 
\end{equation}
The singular locus is the projection by $\pi$ of these solution.
We can rewrite the equation (\ref{sing-eq-b}) as 
\begin{equation} \label{sing-eq-b2}
x_i(1-\varepsilon_i x_i) \xi_i + \sum_{1\leq k \leq m, k \neq i} \varepsilon_i x_k \xi_k = 0 \quad (i = 1, \ldots, m, \varepsilon_i \in \{0, 1\}).
\end{equation}
We fix an $\varepsilon = (\varepsilon_1, \ldots, \varepsilon_m) \in \{0,1\}^m$.
The equation (\ref{sing-eq-b2}) is 
$$
\begin{pmatrix}
x_1(1-\varepsilon_1x_1) & \varepsilon_1 x_2      & \cdots & \varepsilon_1 x_m \\
\varepsilon_2 x_1       &x_2(1-\varepsilon_2x_2) & \cdots & \varepsilon_2 x_m \\
\vdots                  &                       &        &   \vdots    \\
\varepsilon_m x_1       & \varepsilon_m x_2      & \cdots & x_m(1-\varepsilon_m x_m) 
\end{pmatrix}
\begin{pmatrix}
\xi_1 \\
\xi_2 \\
\vdots \\
\xi_m 
\end{pmatrix}
= 
\begin{pmatrix}
0\\
0\\
\vdots\\
0
\end{pmatrix}.
$$
We denote the coefficient matrix by $A_\varepsilon$.
The equation (\ref{sing-eq-b2}) has a solution $(x_1, \ldots, x_m, \xi_1, \ldots, \xi_m)$ with 
$(\xi_1, \ldots, \xi_m) \neq (0, \ldots, 0)$ 
if and only if 
${\rm det}(A_\varepsilon) = 0$ holds.
The defining polynomial for the singular locus is $\prod_{\varepsilon \in \{0,1\}^m}{\rm det}(A_\varepsilon)$.

We compute $\prod_{\varepsilon \in \{0,1\}^m}{\rm det}(A_\varepsilon)$.
We set 
$$
f_m(x_1, \ldots, x_m ; \varepsilon_1, \ldots, \varepsilon_m) = 
{\rm det}
\begin{pmatrix}
1-\varepsilon_1x_1  & \varepsilon_1       & \cdots & \varepsilon_1 \\
\varepsilon_2       & 1-\varepsilon_2x_2  & \cdots & \varepsilon_2 \\
\vdots              & \vdots              &        &   \vdots    \\
\varepsilon_m       & \varepsilon_m       & \cdots & 1-\varepsilon_m x_m 
\end{pmatrix}.
$$ 
Then, the relation 
$ 
{\rm det}(A_\varepsilon) = x_1 \cdots x_m f_m(x_1, \ldots, x_m ; \varepsilon_1, \ldots, \varepsilon_m)
$
holds.
We have 
\begin{align*}
f_m(x_1, \ldots, x_m ; 0, \varepsilon_2, \ldots, \varepsilon_m) =& 
  f_{m-1}(x_2, \ldots, x_m ; \varepsilon_2, \ldots, \varepsilon_m), \\ 
f_m(x_1, \ldots, x_m ; \varepsilon_1, 0, \varepsilon_3, \ldots, \varepsilon_m) =& 
  f_{m-1}(x_1, x_3, \ldots, x_m ; \varepsilon_1, \varepsilon_3, \ldots, \varepsilon_m), \\
& \vdots \\ 
f_m(x_1, \ldots, x_m ; \varepsilon_1, \ldots, \varepsilon_{m-1}, 0) =& f_{m-1}(x_1, \ldots, x_{m-1} ; \varepsilon_1, \ldots, \varepsilon_{m-1}), \\ 
f_m(x_1, \ldots, x_m ; 1, \ldots, 1) = &(-1)^{m-1}((1-x_1)x_2 \cdots x_m + x_1 x_3 \cdots x_m + \cdots \\
  &+ x_1 \cdots x_{m-1}).
\end{align*}
By using these relation, we have
\begin{align*}
& \prod_{\varepsilon \in \{0,1\}^m} f_m(x_1, \ldots, x_m ; \varepsilon_1, \ldots, \varepsilon_m) \\
&=\prod_{\#\{1 \leq i \leq m|\varepsilon_i \neq 0 \}= 1} f_m \prod_{\#\{1 \leq i \leq m|\varepsilon_i \neq 0\}=2} f_m \cdots 
  \prod_{\#\{1 \leq i \leq m|\varepsilon_i \neq 0\} = m} f_m \\
&=\prod_{1 \leq i_1 \leq m} f_1(x_{i_1};1) \prod_{1 \leq i_1 < i_2 \leq m} f_2(x_{i_1}, x_{i_2};1,1) \cdots f_m(x_1, \ldots, x_m; 1, \ldots, 1) \\
&=\prod_{1 \leq i_1 \leq m} (1-x_{i_1}) \prod_{1 \leq i_1 < i_2 \leq m} (x_{i_1} x_{i_2} - x_{i_1} - x_{i_2}) \cdots \\ 
&\quad \quad (-1)^{m-1}(-x_1 x_2 \cdots x_m + x_2 \cdots x_m + \cdots + x_1 \cdots x_{m-1}).
\end{align*}
These gives the following conclusion.
\begin{theorem}
The singular locus of $F_B$ is 
\begin{align*}
{\rm Sing}(I_B(m)) =& {\bf V}(x_1 \cdots x_m \prod_{1 \leq i_1 \leq m}(1-x_{i_1}) 
  \prod_{1\leq i_1 < i_2 \leq m}(x_{i_1} x_{i_2} - x_{i_1} - x_{i_2}) \cdots  \\
& (x_1 x_2 \cdots x_m - x_2 \cdots x_m - \cdots -x_1 \cdots x_{m-1})).
\end{align*}
\end{theorem}

\subsection{Singular locus of Lauricella's system $I_A(m)$}
We compute the singular locus of Lauricella's system $I_A(m)$.
In this case, the computation is not straightforward as the case of $I_B(m)$ and 
we need a Gr\"{o}bner basis in the ring $\mathcal{D}$.
We define a weight vector $\WZeroOne = (0, \ldots, 0, 1, \ldots, 1) \in {\mathbb{Z}}^{2m}$. 
The $\WZeroOne$ initial form of $\ell_i^A$ is 
$$ \innZeroOne(\ell_i^A) = x_i \xi_i \left(x_i\xi_i -x_i \sum_{1 \leq j \leq m} x_j \xi_j \right).$$
We denote the initial form by $L_i^A$.
Since $L_i^A$ is an element in the $\WZeroOne$ initial form ideal $\innZeroOne(I_A(m))$, 
we have
$$ \langle L_1^A, \ldots, L_m^A \rangle \subset \innZeroOne(I_A(m)).$$
It holds that
$$ {\bf V}(L_1^A, \ldots, L_m^A) \supset {\rm Ch}(I_A(m)).$$
For the singular locus, we have 
$$ \pi({\bf V}(L_1^A, \ldots, L_m^A) \setminus \{\xi_1= \cdots = \xi_m = 0\}) \supset {\rm Sing}(I_A(m)).$$
We compute  
$\pi({\bf V}(L_1^A, \ldots, L_m^A) \setminus \{\xi_1= \cdots = \xi_m = 0\})$.
In the analogous way as $I_B(m)$, 
we compute the solutions $(x_1, \ldots, x_m, \xi_1, \ldots, \xi_m)$ with 
$(\xi_1, \ldots, \xi_m) \neq (0, \ldots, 0)$ for  
$$ L_1^A = 0, \ldots, L_m^A = 0 $$ 
i.e.,  
\begin{equation} \label{sing-eq-a}
x_i \xi_i = 0 \text{~ or ~} x_i \xi_i - x_i \sum_{1\leq k \leq m} x_k \xi_k = 0 \quad (i = 1, \ldots, m). 
\end{equation} 
The projection by $\pi$ of these solutions  is 
$\pi({\bf V}(L_1^A, \ldots, L_m^A) \setminus \{\xi_1= \cdots = \xi_m = 0\})$.
The equation (\ref{sing-eq-a}) is rewritten as 
\begin{equation} \label{sing-eq-a2}
x_i \xi_i - x_i \varepsilon_i \sum_{1\leq k \leq m} x_k \xi_k = 0 \quad (i = 1, \ldots, m, \varepsilon_i \in \{0, 1\}). 
\end{equation}
We fix an
$\varepsilon = (\varepsilon_1, \ldots, \varepsilon_m) \in \{0,1\}^m$.
The equation (\ref{sing-eq-a2}) is
$$
\begin{pmatrix}
x_1(1-\varepsilon_1x_1) & -\varepsilon_1 x_1 x_2 & \cdots & -\varepsilon_1 x_1 x_m \\
-\varepsilon_2 x_1 x_2  &x_2(1-\varepsilon_2x_2) & \cdots & -\varepsilon_2 x_2 x_m \\
\vdots                  &\vdots                  &        &   \vdots    \\
-\varepsilon_m x_1 x_m  & -\varepsilon_m x_2 x_m & \cdots & x_m(1-\varepsilon_m x_m) 
\end{pmatrix}
\begin{pmatrix}
\xi_1 \\
\xi_2 \\
\vdots \\
\xi_m 
\end{pmatrix}
= 
\begin{pmatrix}
0\\
0\\
\vdots\\
0
\end{pmatrix}.
$$
We denote the coefficient matrix by $B_\varepsilon$.
The equation (\ref{sing-eq-a2}) has a solution $(x_1,\ldots, x_m, \xi_1, \ldots, \xi_m)$ with 
 $(\xi_1, \ldots, \xi_m) \neq (0, \ldots, 0)$ 
if and only if 
${\rm det}(B_\varepsilon) = 0$ holds. 
So we have
$$
\pi({\bf V}(L_1^A, \ldots, L_m^A) \setminus \{\xi_1= \cdots = \xi_m = 0\}) = \prod_{\varepsilon \in \{0,1\}^m}{\rm det}(B_\varepsilon).
$$ 
We compute the determinant.
\begin{align*}
{\rm det}(B_\varepsilon) =&
x_1 \cdots x_m {\rm det}
\begin{pmatrix}
1-\varepsilon_1x_1 & -\varepsilon_1 x_1  & \cdots & -\varepsilon_1 x_1  \\
-\varepsilon_2 x_2  &1-\varepsilon_2x_2 & \cdots & -\varepsilon_2 x_2  \\
\vdots              & \vdots                       &        &   \vdots    \\
-\varepsilon_m x_m  & -\varepsilon_m x_m & \cdots & 1-\varepsilon_m x_m 
\end{pmatrix} \\
=&
x_1 \cdots x_m {\rm det}
\begin{pmatrix}
1-\varepsilon_1x_1 & -1      & -1      & \cdots & -1  \\
-\varepsilon_2 x_2 & 1       & 0       & \cdots & 0  \\
-\varepsilon_3 x_3 & 0       & 1       & \cdots & 0  \\
\vdots             & \vdots  &  \vdots &      &   \vdots    \\
-\varepsilon_m x_m & 0       & 0       & \cdots & 1 
\end{pmatrix} \\
=& x_1 \cdots x_m (1 - \varepsilon_1 x_1 - \varepsilon_2 x_2 - \cdots - \varepsilon_m x_m)
\end{align*}
So we obtain 
\begin{align*}
\prod_{\varepsilon \in \{0,1\}^m}{\rm det}(A_\varepsilon)
=& x_1^{2^m} \cdots x_m^{2^m} \prod_{1 \leq i_1 \leq m} (1-x_{i_1}) 
  \prod_{1 \leq i_1 < i_2 \leq m} (1 - x_{i_1} - x_{i_2}) \cdots \\
 &(1 - x_1 - \cdots - x_m).
\end{align*}
This equation is the defining polynomial of 
$\pi({\bf V}(L_1^A, \ldots, L_m^A) \setminus \{\xi_1= \cdots = \xi_m = 0\})$. 
We have
\begin{align*}
{\rm Sing}(I_A(m)) \subset& {\bf V}(x_1 \cdots x_m \prod_{1 \leq i_1 \leq m} (1-x_{i_1}) 
  \prod_{1 \leq i_1 < i_2 \leq m} (1 - x_{i_1} - x_{i_2}) \cdots \\
  &(1 - x_1 - \cdots - x_m))  
\end{align*}

We prove the reverse inclusion.
We compute the singular locus in an open neighborhood of the origin.
By Theorem \ref{thm-gr-a},
a Gr\"{o}bner basis of $\mathcal{I}_A(m)$ with respect to $\OrdZeroOne'$ is 
$\{\ell_1^A, \ldots, \ell_m^A\}$.
For a $\Dan$ ideal $\mathcal{I}$, 
the $\WZeroOne$ local initial form ideal is defined by the $\C\{x_1, \ldots, x_m\}[\xi_1, \ldots, \xi_m]$ ideal 
$$ \innZeroOne(\mathcal{I}) = \langle \innZeroOne(P) ~|~ P \in \mathcal{I} \rangle.$$
By the property of a Gr\"{o}bner basis with respect to $\OrdZeroOne'$ 
\cite[Theorem in Section 2]{OakuSingLocus}, 
the $\WZeroOne$ local initial form ideal $\innZeroOne(\mathcal{I}_A(m))$ are generated by 
$L_1^A, \ldots, L_m^A$.
By the theorem in \cite[Th 4.1.]{OSSingLocus},  we obtain the following fact.
\begin{proposition}
For the characteristic variety for the $D$ ideal $I_A(m)$, 
there exists an open neighborhood $U$ of $\{x_1 = \cdots = x_m = 0\}$ in $\C^{2m}$ such that 
${\rm Ch}(I_A(m)) \cap U = {\bf V}(L_1^A, \ldots, L_m^A) \cap U$.
\end{proposition}
We compute the solutions $(x_1, \ldots, x_m, \xi_1, \ldots, \xi_m)$ with
$(\xi_1, \ldots, \xi_m) \neq (0, \ldots, 0)$ for $L_1^A = 0, \ldots, L_m^A = 0$.
The projection by $\pi$ of these solutions
in an open neighborhood of the origin  
is the singular locus in the open neighborhood.
The equations $L_i^A = 0 ~ (i = 1, \cdots m)$ have a solution 
$(x_1, \ldots, x_m, \xi_1, \ldots, \xi_m)$ with $(\xi_1, \ldots, \xi_m) \neq (0, \ldots, 0)$
if and only if 
\begin{align*} 
\prod_{\varepsilon \in \{0,1\}^m}{\rm det}(A_\varepsilon)
 &= x_1^{2^m} \cdots x_m^{2^m} \prod_{1 \leq i_1 \leq m} (1-x_{i_1}) 
  \prod_{1 \leq i_1 < i_2 \leq m} (1 - x_{i_1} - x_{i_2}) \cdots  \\
 & \quad (1 - x_1 - \cdots - x_m)
\end{align*}
holds.
Since we are considering an open neighborhood of the origin, 
the factors 
$$1-x_{i_1}, 1 - x_{i_1} - x_{i_2}, \ldots, 1 - x_1 - \cdots - x_m$$
are not $0$.
We have  
$$ {\rm Sing}(I_A(m)) \cap W = {\bf V}(x_1 \cdots x_m) \cap W, $$
where $W$ is an open neighborhood of the origin in $\C^m$.

Next, we compute the singular locus ${\rm Sing}(I_A(m))$ in the complex torus $(\C^*)^m$.
We assume $x_i \neq 0 ~ (i = 1, \ldots, m)$ i.e., $(x_1, \ldots, x_m) \in (\C^*)^m$. 
We apply the change of coordinates $X_i = \frac{1}{x_i} ~ (i = 1, \ldots, m)$.
By the change of coordinates, the differential operator $\ell_i^A$ changes to   
$$ p_i^A  = X_i \theta_{X_i} (-\theta_{X_i} + c_i - 1) + (\theta_{X_i} - a)(\theta_{X_i} - b_i).$$
We set the $D$ ideal 
$$ I_A(m)' = D \cdot \{p_1^A, \ldots, p_m^A\}.$$
Here, we set $D = \C[X_1, \ldots, X_m]\langle \p_{X_1}, \ldots, \p_{X_m} \rangle$,
i.e., we change the variables $x_i, \p_i$ for $X_i, \p_{X_i}$.
\begin{proposition} \label{prop-gr-a2}
We use the term order $\OrdZeroOne$ defined in Theorem \ref{thm-gr-b}. 
Here, we change the variable $x_i, \p_i$ for $X_i, \p_{X_i}$. 
The set $\{p_1^A, \ldots, p_m^A \}$ is 
a Gr\"{o}bner basis for the $D$ ideal $I_A(m)'$ with respect to $\OrdZeroOne$.
\end{proposition}
We can analogously prove the proposition as Theorem \ref{thm-gr-b}. 
We compute the singular locus of ${\rm Sing}(I_A(m)')$.
By the property of a Gr\"{o}bner basis with respect to $\OrdZeroOne$, 
the $\WZeroOne$ initial form ideal $\innZeroOne(I_A(m)')$ is generated by 
$\innZeroOne(p_1^A),  \ldots, \innZeroOne(p_m^A)$.
The $\WZeroOne$ initial form of $p_i^A$ is 
$$ \innZeroOne(p_i^A) = X_i \xi_{X_i} \left(X_i(1-X_i)\xi_{X_i} + \sum_{1 \leq j \leq m, j \neq i} X_j \xi_{X_j} \right).$$ 
We denote it by $P_i^A$.
Since $P_i^A$ is equal to $L_i^B$, 
$\innZeroOne(I_A(m)')$ is equal to $\innZeroOne(I_B(m))$.
By the result of ${\rm Sing}(I_B(m))$, we have   
\begin{align*}
{\rm Sing}(I_A(m)') =& {\bf V}(X_1 \cdots X_m \prod_{1 \leq i_1 \leq m}(1-X_{i_1})  
  \prod_{1\leq i_1 < i_2 \leq m}(X_{i_1} X_{i_2} - X_{i_1} - X_{i_2})  \\
& \cdots (X_1 X_2 \cdots X_m - X_2 \cdots X_m - \cdots -X_1 \cdots X_{m-1})). 
\end{align*}
The relation $X_i \neq 0$ holds and we apply the change of coordinates $X_i = \frac{1}{x_i}$, 
we have 
\begin{align*}
{\rm Sing}(I_A(m)) \cap (\C^*)^m =& {\bf V}(\prod_{1 \leq i_1 \leq m}(1 - x_{i_1}) 
  \prod_{1\leq i_1 < i_2 \leq m}(1 - x_{i_1} - x_{i_2}) \cdots \\
& (1 - x_1 - \cdots -x_m)). 
\end{align*}

Since 
${\rm Sing}(I_A(m)) \cap W = {\bf V}(x_1 \cdots x_m) \cap W$ 
($W$ is an open neighborhood of the origin) holds, 
we have 
\begin{align*}
{\rm Sing}(I_A(m)) \supset& {\bf V}(x_1 \cdots x_m \prod_{1 \leq i_1 \leq m}(1 - x_{i_1}) 
  \prod_{1\leq i_1 < i_2 \leq m}(1 - x_{i_1} - x_{i_2}) \cdots  \\
& (1 - x_1 - \cdots -x_m)).
\end{align*}
We can prove the reverse inclusion.
\begin{theorem}
The singular locus of $F_A$ is 
\begin{align*}
{\rm Sing}(I_A(m)) =& {\bf V}(x_1 \cdots x_m \prod_{1 \leq i_1 \leq m}(1 - x_{i_1}) 
  \prod_{1\leq i_1 < i_2 \leq m}(1 - x_{i_1} - x_{i_2}) \cdots  \\
& (1 - x_1 - \cdots -x_m)).
\end{align*}
\end{theorem} 

\begin{remark}
In the analogous way as Lauricella's system $I_A(m)$, 
we can also determine the singular locus of Lauricella's system $I_C(m)$. 
Hattori and Takayama \cite{HT} determined the singular locus of the system $I_C(m)$, 
but our method is simpler than their method.
\end{remark}

\section*{Acknowledgements}
The author wishes to express his thanks to Prof. Takayama and Prof. Matsumoto for helpful 
comments concerning Remark \ref{rem-gr-b}. 
This work was partially supported by JSPS KAKENHI Grant Number 24740064.

\end{document}